\documentclass{article}
\usepackage{amsmath,amsfonts,amssymb,amsthm}
\usepackage[round]{natbib}

\newcommand{\real}{\mathbb{R}}
\newcommand{\dnorm}{\mathcal{N}}
\newcommand{\sign}{\mathrm{sign}}
\newcommand{\tran}{\mathsf{T}}

\newcommand{\rd}{\,\mathrm{d}}

\newcommand{\wt}{\widetilde}

\newtheorem{theorem}{Theorem}
\newtheorem{lemma}{Lemma}
\theoremstyle{definition}
\newtheorem{assumption}{Assumption}

\date{January 2014}
\title{The sign of the logistic 
regression coefficient}
\author{A. B. Owen\\Stanford University
\and
P. A. Roediger\\UTRS, Inc.}
\begin{document}

\maketitle
\begin{abstract}
Let $Y$ be a binary random variable and
$X$ a scalar. Let $\hat\beta$ be the maximum
likelihood estimate of the slope in a logistic
regression of $Y$ on $X$ with intercept.
Further let $\bar x_0$ and $\bar x_1$ be the average of sample $x$
values for cases with $y=0$ and $y=1$, respectively.
Then under a condition that rules out separable
predictors, we show that $\sign(\hat\beta) = \sign(\bar x_1-\bar x_0)$.
More generally, if $x_i$ are vector valued then we
show that $\hat\beta=0$ if and only if $\bar x_1=\bar x_0$.
This holds for logistic regression and also for more
general binary regressions with inverse link functions
satisfying a log-concavity condition.
Finally, when $\bar x_1\ne \bar x_0$ then
the angle between $\hat\beta$ and $\bar x_1-\bar x_0$
is less than ninety degrees in binary regressions
satisfying the log-concavity condition and the separation
condition, when the design matrix has full rank.
\end{abstract}

\section{Introduction}

This short note is to introduce and prove
an interesting fact about logistic
regression with a scalar predictor $x$ and
an intercept.  The fact is that the
sign of the slope coefficient's maximum
likelihood estimate (MLE), be it positive, negative
or zero, matches the sign of the difference
in sample means of the predictors. 

This finding about signs was put forth as 
a conjecture in a discussion \citep{ray:roed:neye:2013}
of the paper by \cite{wu:tian:2013} on sensitivity
experiments.  
These sensitivity experiments are sequential experimental 
designs to estimate quantities such as a dose
with 50\% lethality (LD50) in toxicology. 
Other applications, such as safety and reliability
of explosives, require estimates of $x$ with
$\Pr(Y=1\mid X=x)$ much closer to $0$ or $1$.
Every once in a while a chance pattern will lead
to a negative coefficient for a predictor $x$ when
it is known scientifically that $\Pr(Y=1\mid X=x)$
can only increase with $x$.
\cite{ray:roed:neye:2013} advocate continued
testing in this circumstance and remark
that the mean difference is a simple way to detect it.
Similarly, the conjecture would allow one to constrain
the slope's MLE to be positive, by the simple device of
adding an artificial data point $(x_0,0)$
with a very small $x_0$, and/or $(x_{n+1},1)$
with a very large $x_{n+1}$.
Their conjecture can be proved using
elementary methods. The finding is interesting
and does not seem to be widely known. 

It is intuitively clear that a positive coefficient
should be more likely when $\bar x_1$, the average $x$ value for $y=1$,
is larger than $\bar x_0$, the average
$x$ value for $y=0$. But
the pattern is absolute; there
can be no exceptions stemming from
different variances, skewnesses or outliers among the $x$ values.

Logistic regression with fixed $x_i$
and random $y_i$ is an exponential family model.
Given $x_1,\dots,x_n$, the sufficient statistic
is the pair of values $n_1 = \sum_{i=1}^ny_i$ and $\sum_ix_iy_i$. 
\citep[Chapter 2.2.4]{mccu:neld:1989}.
As a result, the logistic regression coefficients
are determined by $x_1,\dots,x_n$ along with $n_1$
and $\bar x_1$.
Given $x_i$, the sufficient statistics can
be used to compute $(n_0,n_1,\bar x_0,\bar x_1)$
and vice versa, where $n_0=n-n_1$.
But this latter quadruple does not
determine the MLE on its own.  The sufficiency
is only a conditional one and does
not quite explain the sign result.
Moreover, probit regression models do not have
such sufficient statistics, yet we find that their slope 
also has a sign determined 
by that of $\bar x_1-\bar x_0$, without regard
to other features of the $x_i$ sample. The sign
result stems from log concavity of the logistic 
and Gaussian density functions
and it extends to many other binary regressions.

Section \ref{sec:notation} introduces the notation
we need. Section~\ref{sec:proof} has our elementary proof
and Section~\ref{sec:vector} considers two
generalizations for vector-valued predictors
$x_i\in\real^d$. Both generalizations make
some mild assumptions about the configuration
of $x_i$ values.  In binary regressions, the
inverse of the link function is usually
a cumulative distribution function (CDF).
When that CDF corresponds to a log concave
probability density function, we find
that the coefficient of $x$ is zero if and only
if the mean $x$-values for $y=0$ and $y=1$ coincide.
When the means do not coincide, the regression coefficient
makes less than a $90$ degree
angle with the difference in $x$ means.

We conclude this section by mentioning some similar results.
When $X\in\real^d$  are independently $\dnorm(\mu_y,\Sigma)$ distributed
conditionally on $Y=y$, for a covariance matrix $\Sigma$ of full rank,
then the population version of the logistic
regression coefficient is $\beta = \Sigma^{-1}(\mu_1-\mu_0)$.
In this case, $\beta=0$ if and only if $\mu_1=\mu_0$.
Furthermore, when $\mu_1\ne\mu_0$ then
$\beta^\tran(\mu_1-\mu_0)>0$.
In some infinitely imbalanced limits where $n_0\to\infty$ while
$n_1$ remains fixed, the MLE of the slope coefficient
depends on $x_i$ for $y_i=1$ only through
$\bar x_1$ \citep{infimbal}. This work grew from a
correspondence about using the result in \cite{infimbal}
to address the conjecture in \cite{ray:roed:neye:2013}.

\section{Notation and basic result}\label{sec:notation}

In the scalar predictor case,
the data are $(x_i,y_i)$ for $i=1,\dots,n$
with $x_i\in\real$ and $y_i\in\{0,1\}$.
There are $n_0$ observations with $y_i=0$
and $n_1$ with $y_i=1$. To avoid trivial
complications, we assume that $\min(n_0,n_1)>0$.

Let $\bar x_1 = (1/n_1)\sum_{i=1}^nx_iy_i$
and $\bar x_0 = (1/n_0)\sum_{i=1}^nx_i(1-y_i)$
be the sample averages of $x$ for observations
with $y_i=1$ and $y_i=0$, respectively.
A logistic regression model has
$$\Pr( Y=1\mid X=x) = 
\frac{\exp(\alpha + \beta x)}
{1+\exp(\alpha + \beta x)}\equiv p(x;\alpha,\beta).$$
The likelihood function is
\begin{align}\label{eq:likelihood}
L(\alpha,\beta) = 
\prod_{i=1}^n p(x_i;\alpha,\beta)^{y_i}
(1-p(x_i;\alpha,\beta))^{1-y_i}.
\end{align}
This model has a well defined maximum likelihood
estimate if the $x$ data for $y=0$ overlap
sufficiently with those for $y=1$. 
In Section~\ref{sec:vector} we state
the overlap conditions given by \cite{silv:1981}.

For scalar $x$, Silvapulle's conditions simplify.
Let 
$L_0 = \min\{x_i\mid y_i=0\}$,
$U_0 = \max\{x_i\mid y_i=0\}$,
$L_1 = \min\{x_i\mid y_i=1\}$, and
$U_1 = \max\{x_i\mid y_i=1\}$ be the
extreme values of $x$ in each of the
two groups.
It is sufficient to have
\begin{align}\label{eq:olap1}
L_0 < U_1\ \&\ L_1 < U_0.
\end{align}
If the intervals $[L_0,U_0]$ and $[L_1,U_1]$
overlap in an interval of positive length, then~\eqref{eq:olap1}
is satisfied.  That is the usual case when
logistic regression is used, but other corner cases
satisfy Silvapulle's condition too.
For instance, if all the $x$'s for one $y$ value,
say $y=1$, are identical, then~\eqref{eq:olap1} can still
hold so long as $L_0<L_1=U_1<U_0$.

We can even weaken~\eqref{eq:olap1} to allow the $x$
values for one group to form a zero-length interval
tied with an extreme value from the other group:
\begin{equation}\label{eq:olap2}
\begin{split}
L_0&=U_0=L_1<U_1
\quad\text{or}\quad 
L_1<U_1=L_0=U_0
\quad\text{or}\\
L_1&=U_1=L_0<U_0
\quad\text{or}\quad 
L_0<U_0=L_1=U_1.
\end{split}
\end{equation}
We cannot however have $L_0=L_1=U_0=U_1$. 
For scalar $x$ with $n_0>0$ and $n_1>0$,
Silvapulle's conditions are equivalent 
to \eqref{eq:olap1} or \eqref{eq:olap2}.

The sign function we use is defined for $z\in\real$ by
$\sign(z)=1$ for $z>0$, $\sign(z) = -1$ for $z<0$, 
and $\sign(0)=0$.
Our first result is the following:
\begin{theorem}\label{thm:mainone}
Let $x_i\in \real$ and $y_i\in\{0,1\}$
for $i=1,\dots,n$. Assume that both $n_1=\sum_{i=1}^ny_i>0$ 
and $n_0=n-n_1>0$ and that $x_i$ and $y_i$
satisfy an overlap condition~\eqref{eq:olap1}
or~\eqref{eq:olap2}.
Then the likelihood~\eqref{eq:likelihood} has
a unique maximizer
$(\hat\alpha,\hat\beta)$ 
with $\sign(\hat\beta) = \sign(\bar x_1-\bar x_0)$.
\end{theorem}

The conclusion of Theorem~\ref{thm:mainone}
still holds in some cases where neither~\eqref{eq:olap1}
nor~\eqref{eq:olap2}  hold,
though it may require some interpretation.
For instance, if  $U_0<L_1$,
then $\bar x_1>\bar x_0$ and also $\hat\beta = +\infty$.
Likewise, if  $U_1<L_0$,
then $\bar x_1<\bar x_0$ with $\hat\beta = -\infty$.
So these two cases are included if we take
$\sign(\pm\infty)=\pm1$.
An exception arises
when all $x_i$ have
the same value. Then $\bar x_1-\bar x_0=0$ but
the likelihood has no unique maximizer.

\section{Proof of Theorem~\ref{thm:mainone}}\label{sec:proof}

The existence of a unique maximizer 
for logistic regression here follows from
the theorem in \cite{silv:1981}. So we only
need to consider the sign of $\hat\beta$.

The log likelihood in the logistic regression is
$$\ell(\alpha,\beta)
=\sum_{i=1}^ny_i(\alpha+\beta x_i)-\log(1+\exp(\alpha+\beta x_i)).$$
This is a concave function of the parameter $(\alpha,\beta)$.
The maximum likelihood estimates $(\hat\alpha,\hat\beta)$ are 
attained by setting
\begin{align}\label{eq:foralpha}
0  = \frac{\partial \ell}{\partial\alpha}
 = \sum_{i=1}^n (y_i - p(x_i;\alpha,\beta))
\end{align}
and
\begin{align}\label{eq:forbeta}
0  = \frac{\partial \ell}{\partial\beta}
 = \sum_{i=1}^n (y_i - p(x_i;\alpha,\beta))x_i.
\end{align}
We will abbreviate $p(x_i;\hat\alpha,\hat\beta)$
to $\hat p_i$.
From equation~\eqref{eq:foralpha} we find
that $\bar p \equiv (1/n)\sum_{i=1}^n \hat p_i = n_1/n$.
This is a well-known consequence of including
an intercept term in logistic regression.

Subtracting $\bar x$ times equation~\eqref{eq:foralpha}
from equation~\eqref{eq:forbeta} yields
$
\sum_{i=1}^n (y_i - \hat p_i)(x_i-\bar x)=0.
$
After  rearranging the sum we have
\begin{align}\label{eq:penultimate}
n_1(\bar x_1-\bar x) = \sum_{i=1}^n\hat p_i(x_i-\bar x).
\end{align}

On the left of~\eqref{eq:penultimate}
we find that $\sign(\bar x_1-\bar x) 
= \sign( \bar x_1 - (n_1/n)\bar x_1-(n_0/n)\bar x_0)
= \sign( (n_0/n)(\bar x_1 - \bar x_0)) = \sign(\bar x_1-\bar x_0)$.
For the right side, we consider three cases.
First, if $\hat\beta=0$, then $\hat p_i$ is constant
and the right side of~\eqref{eq:penultimate} is zero.

Second, if $\hat\beta>0$, then $\hat p_i$ is a strictly increasing
function of $x_i$. Then
\begin{align}\label{eq:second}
\sum_{i=1}^n \hat p_i (x_i-\bar x)
=\sum_{i=1}^n (\hat p_i -\wt p)(x_i-\bar x),
\end{align}
where $\wt p = 
\exp(\hat\alpha+\hat\beta\bar x)
/(1+\exp(\hat\alpha+\hat\beta\bar x))$.
Each term on the right of~\eqref{eq:second} is a product
of two positive numbers, two negative numbers
or two zeros.  Therefore~\eqref{eq:second}
cannot be negative.
Because the $x_i$ cannot all be 
equal under either~\eqref{eq:olap1} or~\eqref{eq:olap2},
at least two of the
terms in~\eqref{eq:second} must be strictly positive.  
Therefore the right side of
\eqref{eq:penultimate} is positive when $\hat\beta>0$.

Similarly if $\hat\beta<0$ (the third case) then the right side
of \eqref{eq:penultimate} is negative. In all three
cases, the sign of $\hat\beta$ matches the sign
of the right side of~\eqref{eq:penultimate} and
hence equals the sign of $\bar x_1-\bar x_0$.

\section{Generalizations}\label{sec:vector}

Here we generalize the connection
between $\hat\beta$ and the difference
in group means
for $x_i\in\real$ to some other settings.
The first setting
is to allow $x_i\in\real^d$ for an integer $d\ge 1$.
As before, we let $\bar x_0$ and $\bar x_1$
be the averages of $x_i$ for $y=0$ and for $y=1$
respectively.
The logistic regression model is now
$\Pr( Y=1\mid X=x) = \exp(\alpha + x^\tran\beta)
/(1+\exp(\alpha + x^\tran\beta))$.

The second generalization extends the logistic
model to models of the form
$\Pr( Y=1\mid X=x) = G(\alpha+x^\tran\beta)$
where $G(\cdot)$ is a non-decreasing function
from $\real$ to $[0,1]$.
The function $G^{-1}$ is called the link function.
The link function applied to 
$\Pr(Y=1\mid X=x)$ is an affine function, 
$\alpha + x^\tran\beta$, of $x$.
Besides the logistic model, some important
alternatives are the probit model
with $G(z) = \Phi(z)$ where $\Phi$ is the
CDF of the $\dnorm(0,1)$ distribution,
and the complementary log-log link function
whose inverse is $G(z) = 1-\exp(-\exp(z))$.
See~\cite{mccu:neld:1989}.

\subsection{Assumptions}

Before generalizing to other contexts,
we present conditions that we will need.
We need assumptions about the data and
assumptions about the link function.
For the data, we first extend
$x_i$ to $\wt x_i = (1,x_i^\tran)^\tran$
whose first component creates the
intercept term.

\begin{assumption}[Full rank condition]\label{as:full}
The matrix ${\mathcal X}\in\real^{n\times(d+1)}$
with $i$'th row equal to $\wt x_i^\tran$ has
full rank $d+1\le n$.
\end{assumption}

Assumption~\ref{as:full} is commonly made in regression
settings. When it fails to hold, then one of the component
variables in $x_i$ can be replaced by an affine combination of the
others. In that case, such redundant variables are often dropped from
the model until Assumption~\ref{as:full} holds.


Next we state an assumption that keeps the
maximum likelihood estimates bounded.  This
assumption imposes some overlap among the $x$'s
for which $y=0$ and the ones for which $y=1$.
Using the extended predictors, let
\begin{align}\label{eq:cones}
S = \Biggl\{ \sum_{i: y_i=1}k_i\wt x_i\mid k_i>0\Biggr\},
\quad\text{and}\quad
F = \Biggl\{ \sum_{i: y_i=0}k_i\wt x_i\mid k_i>0\Biggr\}.
\end{align}
These are open convex cones in $\real^{d+1}$ generated by 
the extended predictor values for $y=1$ and $y=0$ respectively.
Silvapulle's overlap condition is that
either $S\cap F \ne\varnothing$ or $S=\real^{d+1}$
or $F=\real^{d+1}$.
The latter two possibilities cannot hold in our
setting with an intercept, so we only
need the first condition, which we label Assumption~\ref{as:silolap}.

\begin{assumption}\label{as:silolap}[Overlap condition]
Let the cones $S$ and $F$ be defined from the
data as at~\eqref{eq:cones}. Then
$S\cap F\ne\varnothing$.
\end{assumption}

The above are the assumptions we need on the
data. Next, we give an assumption for the link function.

\begin{assumption}[Silvapulle's (1981) link condition]\label{as:link}
The inverse link function $G$ 
is strictly increasing
at every value of $z$ with $0<G(z)<1$,
and both $-\log(G(z))$ and $-\log(1-G(z))$ are convex
functions of $z$.
\end{assumption}

Silvapulle's definition of convexity allows
functions that take the value $+\infty$.
For example if $G$ is the CDF of the $U[0,1]$
distribution then both $-\log(G(z))$ and $-\log(1-G(z))$
are convex functions, the former equalling $\infty$
for $z\le0$ and the latter equalling $\infty$
for $z\ge 1$.
There is a typographical error in part (iii)
of the Theorem in \cite{silv:1981}:  it supposes
that $-\log(G)$ and $\log(1-G)$ are convex but
the second one should be $-\log(1-G)$.

Assumption~\ref{as:link} requires both the CDF
$G$ and the survivor function $1-G$
to be log-concave functions of $z$.  A log-concave
function is one whose logarithm is concave.
Many probability density functions are log-concave.
Both the CDF and survivor functions inherit log concavity
from the density function.

\begin{lemma}\label{lem:bagnberg}
Let $g(z)$ be a probability density function for $z\in\real$.
If
$-\log(g(z))$ is convex
then $G(z)=\int_{-\infty}^zg(t)\rd t$ satisfies
Assumption~\ref{as:link}.
\end{lemma}
\begin{proof}
Log concavity of $G$ and $1-G$ both follow from Lemma 3 of \cite{an:1998}.
\cite{an:1998} requires $g$ to be measurable but that holds automatically
for log concave $g$.
Let $z\in\real$
satisfy  $0<G(z)<1$. Then there is a point $a<z$ with $g(a)>0$
and a point $b>z$ with $g(b)>0$.  From log concavity of $g$
we have $g(t)>\min(g(a),g(b))>0$ for all $t$ in the interval $(a,b)$
which contains $z$. 
Therefore $G$ is strictly increasing at $z$.
\end{proof}

Lemma~\ref{lem:bagnberg} makes it easy to identify
a large family of link functions that satisfy
Silvapulle's condition.
If $g(z)$ is a log concave density on $\real$
then it satisfies Assumption~\ref{as:link}
including the requirement to be strictly positive
where the corresponding CDF $G$ satisfies
$0<G(z)<1$.
\cite{bagn:berg:2005} list the following
log-concave densities among others:
uniform, normal, exponential, logistic, extreme value,
double exponential, $cz^{c-1}$ on $(0,1]$ for $c\ge 1$, Weibull
with shape parameter at least $1$, and
the  Gamma distribution with shape at least $1$.

The inverse link for the complementary log-log  model
is easily seen to satisfy Assumption~\ref{as:link}:
It corresponds to the density
$g(z) = \exp( z-\exp(z))$ and $\log(g(z))$ has second
derivative $-\exp(z)<0$.
The Cauchy CDF, $G(z) = (1/\pi)\arctan(z)+1/2$,
has been suggested for binary regression models \citep{morg:smit:1992}.
It is somewhat robust to mislabeling among the $y_i$,
but this $G$ is not log concave.

\subsection{Equivalence of $\hat\beta=0$ and $\bar x_0=\bar x_1$}

The log likelihood function for binary
regression with inverse link $G$ is
\begin{align}\label{eq:loglikG}
\ell(\alpha,\beta) = 
\sum_{i=1}^n y_i \log(G(\alpha+x_i^\tran\beta))
+(1-y_i) \log(1-G(\alpha+x_i^\tran\beta)).
\end{align}

\begin{theorem}\label{thm:zero}
Let $x_i\in\real^d$ and $y_i\in\{0,1\}$ for $i=1,\dots,n$
satisfy the full rank Assumption~\ref{as:full}
and the overlap Assumption~\ref{as:silolap}. 
Let $G$ satisfy the link Assumption~\ref{as:link}.
Then $\bar x_1 = \bar x_0$ if and only
if the model~\eqref{eq:loglikG} has a unique
maximum likelihood estimate
$(\hat\alpha,\hat\beta)$ with $\hat\beta=0$.
\end{theorem}
\begin{proof}
Under these assumptions,
the Theorem in \cite{silv:1981}
yields that there is a unique maximum likelihood
estimate $(\hat\alpha,\hat\beta)$.
It solves the equations
\begin{align}\label{eq:byalpha}
0&=\frac{\partial}{\partial \alpha} \ell(\alpha,\beta)
= \sum_{i=1}^n
y_i \frac{g(\alpha+x_i^\tran\beta)}{G(\alpha+x_i^\tran\beta)}
-(1-y_i) \frac{g(\alpha+x_i^\tran\beta)}{1-G(\alpha+x_i^\tran\beta)}
\end{align}
and
\begin{align}\label{eq:bybeta}
0&=\frac{\partial}{\partial \beta} \ell(\alpha,\beta)
= \sum_{i=1}^n
y_i \frac{g(\alpha+x_i^\tran\beta)}{G(\alpha+x_i^\tran\beta)}x_i
-(1-y_i) \frac{g(\alpha+x_i^\tran\beta)}{1-G(\alpha+x_i^\tran\beta)}x_i
\end{align}
where $g$ is the derivative of $G$.

If the MLE has $\hat\beta=0$,
then we can solve equation~\eqref{eq:byalpha} 
to find that $G(\hat\alpha) = n_1/n$.
Then equation~\eqref{eq:bybeta} yields
$n_1\bar x_1g(\hat\alpha)/G(\hat\alpha)=
n_0\bar x_0g(\hat\alpha)/(1-G(\hat\alpha))$
from which we get $\bar x_1=\bar x_0$.
Conversely, if $\bar x_0=\bar x_1$ then
$\hat\alpha = G^{-1}(n_1/n)$ and $\hat\beta = 0$
jointly solve equations~\eqref{eq:byalpha} and~\eqref{eq:bybeta}
and hence provide the unique maximum likelihood estimate.
\end{proof}

\subsection{Angle between $\hat\beta$ and $\bar x_1-\bar x_0$}

For the special case of logistic regression, 
where $G(z) =\exp(z)/(1+\exp(z))$,
suppose that Assumptions~\ref{as:full} and~\ref{as:silolap}
are satisfied.
Then the maximum
likelihood estimates are well defined.
Using the same argument as in Section~\ref{sec:proof},
but multiplying both sides of \eqref{eq:penultimate} 
by $\hat\beta^\tran$,
we find that $\hat\beta = 0$ if and only
if $\bar x_0=\bar x_1$.
Otherwise $\hat\beta^\tran(\bar x_1-\bar x_0)>0$.
In other words, when $\bar x_1\ne\bar x_0$,
then $\hat\beta$ makes less than
a ninety degree angle with $\bar x_1-\bar x_0$.
The result holds more generally.

\begin{theorem}\label{thm:angle}
Let $x_i\in\real^d$ and $y_i\in\{0,1\}$ for $i=1,\dots,n$
satisfy the full rank Assumption~\ref{as:full}
and the overlap Assumption~\ref{as:silolap}. 
Let $G$ satisfy the link Assumption~\ref{as:link}.
If $\bar x_1-\bar x_0\ne0$ and $(\hat\alpha,\hat\beta)$
maximize the log likelihood~\eqref{eq:loglikG},
then $\hat\beta^\tran(\bar x_1-\bar x_0)>0$.
\end{theorem}
\begin{proof}
We use two data sets. The original and 
a shifted one with $x_i^* = x_i-y_i\Delta$
where $\Delta \equiv \bar x_1-\bar x_0$.
We use $\ell_*$ to denote the log likelihood
of the shifted data set.  In the shifted data set, 
$(1/n_1)\sum_{i=1}^ny_ix_i^*=(1/n_0)\sum_{i=1}^n(1-y_i)x_i^*$
by construction.
The overlap assumption also holds in the shifted data set.
The shifted data set has MLE $\beta_*=0$
and $\alpha_* = G^{-1}(n_1/n)$.
Now suppose to the contrary of the theorem
that $\hat\beta^\tran\Delta \le 0$.
Then
\begin{align*}
\ell(\hat\alpha,\hat\beta) 
&=\sum_{i=1}^ny_i\log( G(\hat\alpha + \hat\beta^\tran x_i))
+(1-y_i)\log( 1-G(\hat\alpha+\hat\beta^\tran x_i))\\
&\le\sum_{i=1}^ny_i\log( G(\hat\alpha + \hat\beta^\tran (x_i-\Delta)))
+(1-y_i)\log( 1-G(\hat\alpha+\hat\beta^\tran x_i))\\
& = \ell_*(\hat\alpha,\hat\beta) \le \ell_*(\alpha_*,0) = \ell(\alpha_*,0).
\end{align*}
As a result, $\hat\beta$ is not the unique MLE of $\beta$
that it would have been, had it maximized~\eqref{eq:loglikG}
under the given assumptions.
The first inequality arises because $G$ is nondecreasing.
The second inequality follows because, from Theorem~\ref{thm:angle},
the maximizers of $\ell_*$ are $(\alpha_*,0)$. 
The full rank assumption may fail to hold in the shifted
data set, but if it does then $(\alpha_*,0)$ is still a maximizer
of $\ell_*$ though not the unique maximizer, and our proof
here does not need uniqueness in the shifted data.
\end{proof}

\section*{Acknowledgments}

We thank Alan Agresti, Brad Efron, Peter McCullagh, Mervyn
Silvapulle and Jeff Wu and Danny Wang for helpful comments.
ABO was supported by the National Science Foundation
under grant DMS-0906056.
PAR was contractually supported by the U.S.\ Army Armament Research, Development and Engineering Center (ARDEC).

\bibliographystyle{apalike}
\bibliography{logistic}

\begin{thebibliography}{}

\bibitem[An, 1998]{an:1998}
An, M.~Y. (1998).
\newblock Logconcavity versus logconvexity: A complete characterization.
\newblock {\em Journal of Economic Theory}, 80:350--369.

\bibitem[Bagnoli and Bergstrom, 2005]{bagn:berg:2005}
Bagnoli, M. and Bergstrom, T. (2005).
\newblock Log-concave probability and its applications.
\newblock {\em Economic Theory}, 26(2):445--469.

\bibitem[McCullagh and Nelder, 1989]{mccu:neld:1989}
McCullagh, P. and Nelder, J.~A. (1989).
\newblock {\em Generalized linear models}.
\newblock Chapman \& Hall, Boca Raton, FL.

\bibitem[Morgan and Smith, 1992]{morg:smit:1992}
Morgan, B. J.~T. and Smith, D.~M. (1992).
\newblock A note on {Wadley's} problem with overdispersion.
\newblock {\em Journal of the Royal Statistical Society, Series C},
  41(2):349--354.

\bibitem[Owen, 2007]{infimbal}
Owen, A.~B. (2007).
\newblock Infinitely imbalanced logistic regression.
\newblock {\em Journal of Machine Learning Research}, 8:761--773.

\bibitem[Ray et~al., 2013]{ray:roed:neye:2013}
Ray, D.~M., Roediger, P.~A., and Neyer, B.~T. (2013).
\newblock Commentary: Three-phase optimal design for sensitivity experiments.
\newblock {\em Journal of Statistical Planning and Inference}.
\newblock In press.

\bibitem[Silvapulle, 1981]{silv:1981}
Silvapulle, M.~J. (1981).
\newblock On the existence of maximum likelihood estimates for the binomial
  response models.
\newblock {\em Journal of the Royal Statistical Society, Series B},
  43:310--313.

\bibitem[Wu and Tian, 2013]{wu:tian:2013}
Wu, C. F.~J. and Tian, Y. (2013).
\newblock Three-phase optimal design of sensitivity experiments.
\newblock {\em Journal of Statistical Planning and Inference}.
\newblock In press.

\end{thebibliography}
\end{document}